\documentclass[a4paper,12pt]{article}
\usepackage{fullpage}

\usepackage{cmbright}
\usepackage{microtype}
\usepackage[T1]{fontenc}
\usepackage{lmodern}
\usepackage[utf8]{inputenc}
\usepackage{bm} 
\usepackage{hwemoji}
\usepackage{amsfonts,amssymb,amsmath,amsthm,dsfont} 
\usepackage{xspace}
\usepackage{graphicx}
\usepackage{xcolor}
\usepackage{cite}
\usepackage{subfig}
\usepackage{comment}

\usepackage{hyperref}
\definecolor{darkgreen}{rgb}{0,0.4,0}
\definecolor{BrickRed}{rgb}{0.65,0.08,0}
\hypersetup{colorlinks=true,linkcolor=blue,citecolor=red,filecolor=BrickRed,urlcolor=darkgreen}

\usepackage[capitalize,nameinlink,noabbrev]{cleveref}

\usepackage{tikz}
\usepackage{tkz-tab}
\usepackage{booktabs}
\usepackage{multirow}
\usepackage{todonotes}
\usepackage{comment}

\newtheorem{theorem}{Theorem}
\newtheorem*{theorem*}{Theorem}
\newtheorem{lemma}[theorem]{Lemma}
\newtheorem{proposition}[theorem]{Proposition}

\newtheorem{open}[theorem]{Open problem}

\newtheorem*{main-thm}{Main Theorem}
\crefname{assumption}{Assumption}{Assumptions}

\theoremstyle{definition}

\newtheorem{remark}[theorem]{Remark}

\newcommand{\eps}{\varepsilon}

\renewcommand{\P}{\mathbb{P}}
\def\Pat{\mathcal P}
\NewDocumentCommand\hidden{}{\scalerel*{\includegraphics{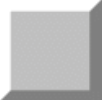}}{X}}
\NewDocumentCommand\mine{}{\scalerel*{\includegraphics{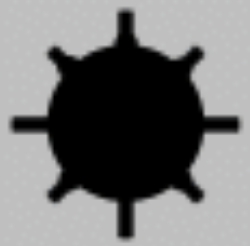}}{X}}
\def\A{\mathcal A}
\def\Z{\mathbb Z}
\def\empty{\varnothing}
\def\flag{🚩}
\def\reveal{\text{Reveal}}
\def\cell{c}
\def\ma{mine assignment}
\def\gs{grid state}
\def\Mnp{M_{n,p}}
\def\cut{n^{-1/6}}
\def\ap{ambiguous pattern}
\def\env{\text{Envelope}}
\def\Sc{S_{\text{clue}}}
\def\Smin{S^{\text{min}}}
\def\cutoff{\kappa_n}

\providecommand{\keywords}[1]
{
  \small	
  \noindent
  \textbf{{Keywords:}} #1
}

\begin{document}

\title{\textbf{Phase transition for Minesweeper}}

\author{%
Baptiste Louf\thanks{CNRS and Institut de Mathématiques de Bordeaux, France. Partially supported by ANR CartesEtPlus (ANR-23-CE48-0018) and ANR HighGG (ANR-24-CE40-2078-01).}}

\maketitle

\begin{abstract}
We prove a coarse phase transition for the game of  Minesweeper: above a certain critical mine density, the game becomes unsolvable with high probability, whereas below the critical mine density it can be solved with a linear time algorithm. 
\thispagestyle{empty}
\end{abstract}

\keywords{Phase transition, minesweeper, random discrete structures}


\section{Introduction}

Minesweeper is a classic computer game that starts with a rectangular grid whose cells have been obscured, under some of which lie mines. The player needs to determine where the mines are by revealing cells (without hitting a mine), and each revealed cell gives a clue about the number of mines in its vicinity.

\begin{figure}
\center
\includegraphics[scale=0.4]{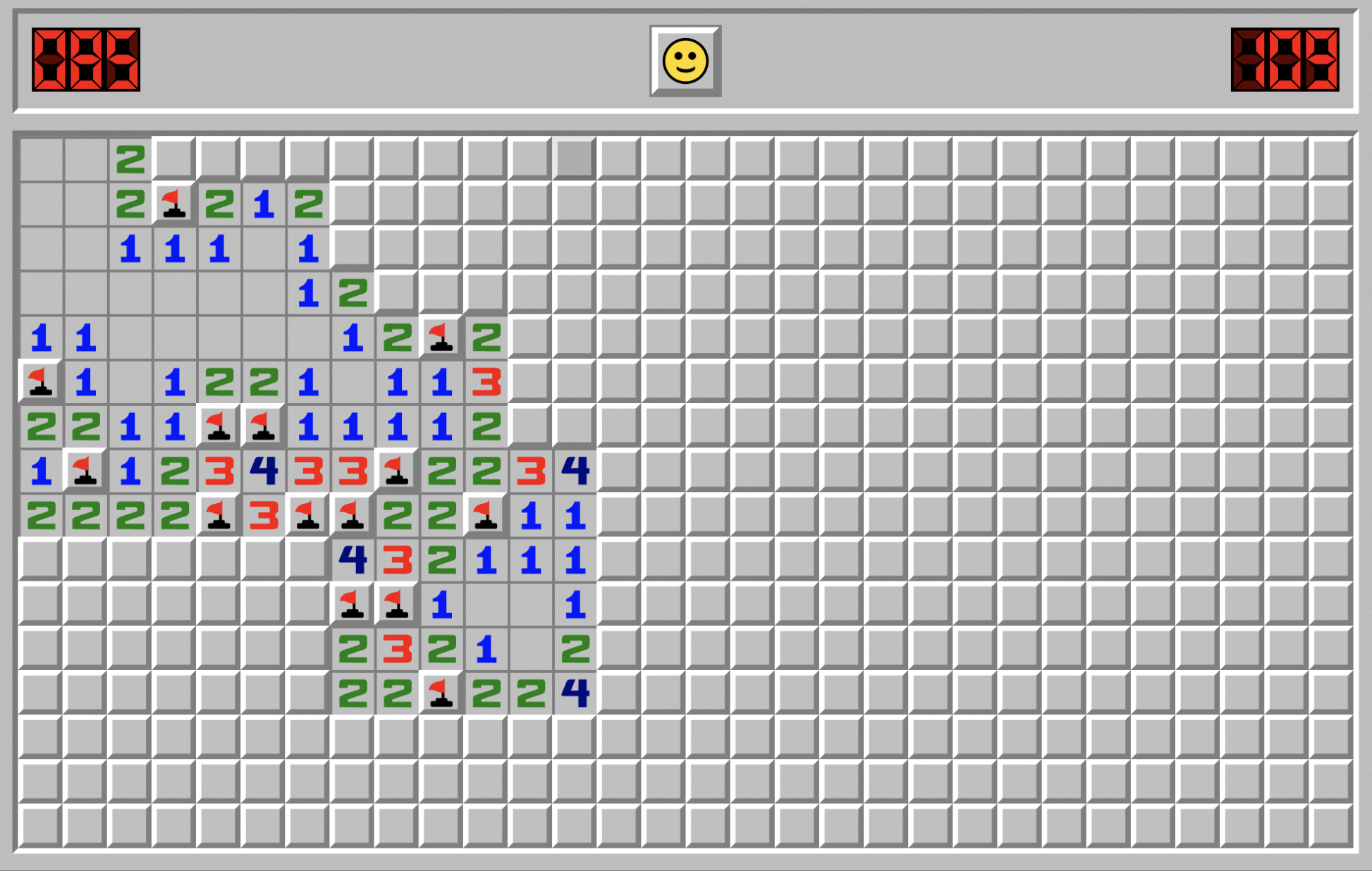}
\caption{An ongoing game of minesweeper. All figures in this paper are screenshots from the website \url{https://minesweeper.us/}.}
\end{figure}

As a mathematical problem, Minesweeper has mostly been studied from the complexity theory perspective: in~\cite{Kaye} it is shown that deciding whether a certain grid state is \emph{consistent} with a mine assignment is NP-complete, and the \emph{inference} problem (i.e. given a grid state, is there a safe cell to click on ?) is coNP-complete (see~\cite{coNP,coNPbis}). 

However, these are worst-case complexity results.  It is well known that many problems that are hard in the worst case become much more tractable when considering \emph{random} instances. A famous example is the $3$-SAT problem (which is NP-complete~\cite{Karp}): consider a random instance with a fixed \emph{density} of clauses. If the density is above a certain threshold, then the instance is most likely unsatisfiable; if it is below the threshold, it is most likely satisfiable—and finding a satisfying variable assignment can be done quickly\footnote{this very simplified presentation doesn't do justice to the richness and complexity of the research on random SAT instances; we refer to \cite{Perkins} for a more detailed account.}.This kind of behavior is referred to as a \emph{phase transition}, and similar phenomena have been observed in many other random structures. We refer to~\cite{KS,Perkins,random_graphs} and references therein for more details.

In this paper, we prove a phase transition result regarding whether a random Minesweeper instance is solvable, with respect to the \emph{mine density} (this question has already been approached experimentally in~\cite{phase}). Note that contrary to most properties where a phase transition is observed, solvability of Minesweeper is \textbf{not} a monotone property. Indeed, adding a mine to an unsolvable instance can make it solvable (see \Cref{prop_not_monotone}).

Let $M_{n,p}$ be the random minesweeper instance on a rectangular grid of size $n$ where each cell receives a mine independently with probability $p$. Our main result is the following.

\begin{theorem}\label{main_thm}
There is a coarse phase transition at $p=\Theta(n^{-1/6})$, i.e.
\begin{itemize}
\item if $p=o(n^{-1/6})$ then there exists a \textbf{linear time} algorithm that solves $\Mnp$ with high probability; 
\item if $p>>n^{-1/6}$ then with high probability no algorithm solves $\Mnp$, i.e.
\[\max_{\text{algorithms }\A}\P(\A\text{ solves }\Mnp)=o(1);\]
\item for every $c>0$, there exists $\eps$ such that, taking $p\sim cn^{-1/6}$ and $n$ large enough,
\[\eps\leq\max_{\text{algorithms }\A}\P(\A\text{ solves }\Mnp)\leq 1-\eps.\]
\end{itemize} 
\end{theorem}

The phase transition is driven by the emergence of small \emph{ambiguous mine patterns}. This is reminiscent of Friedgut's theorem~\cite{Friedgut} which states that for monotone properties, there is a coarse phase transition if and only if the property is well approximated by local constraints. This suggests that solvability of a random minesweeper instance is very close to being monotone. And indeed, we show it is "\emph{almost surely monotone}" in \Cref{sec_process}, where we study a random process version in which mines appear one after the other. In that same section, we also show a "\emph{hitting time property}", i.e. non-solvability coincides with the appearance of the smallest ambiguous pattern. 

The paper is organized this way: in \Cref{sec_def} we give definitions. \Cref{sec_patterns} is the main technical part of the proof, where we identify the smallest ambiguous patterns. Then in \Cref{sec_proof}, we prove the main theorem, using the results of \Cref{sec_patterns} and standard probabilistic arguments. Finally, in \Cref{sec_process}, we study the random mine process, showing a hitting time property and almost sure monotonicity.

\section{Definitions}\label{sec_def}

Given a rectangular grid $G$, we define the \emph{neighborhood} $N(\cell)$ of a cell $\cell$ as the cell itself plus the (up to 8) cells that touch it. This induces a natural distance $d$ on the cells of the grid\footnote{the graph distance in the neighbor graph.}. A \emph{mine assignment} $M$ on $G$ is a function that assigns the value $\empty$ ("\emph{empty}") or $\mine$ ("\emph{mine}") to each cell $\cell\in G$, while a \emph{grid state} $S$ on $G$ is a function that assigns to each cell $\cell\in G$ a value in $\{\hidden,0,1,2,3,4,5,6,7,8,\flag,\mine\}$. We say that $\cell$ is a \emph{clue} if $S(\cell)\in\Z$, a \emph{flag} if  $S(\cell)=\flag$ and \emph{hidden} if $S(\cell)=\hidden$.
 A mine assignment $M$ and a grid state $S$ on the same grid $G$ are said to be \emph{consistent} if for each cell $\cell$, we have either :
\begin{itemize}
\item $S(\cell)=\hidden$;
\item $S(\cell)=|\{\cell'\in N(\cell)|M(\cell')=\mine\}|$;
\item $S(\cell)\in\{\mine,\flag\}$ and $M(\cell)=\mine$.
\end{itemize} 

Finally, if a \ma{} $M$ and a \gs{} $S$ are consistent  with each other, and if $\cell\in S$ is hidden, then let $\reveal(M,S,\cell)$ be the \gs{} that equals $S$ except at $\cell$, where it is equal to $\mine$ if $M(\cell)=\mine$, and to $|\{\cell'\in N(\cell)|M(\cell')=\mine\}|$ otherwise. An algorithm $\mathcal A$ is a function that associates to each \gs{} $S$ one of its hidden cells. Given an algorithm $\A$ and a \ma{} $M$ of the grid $G$, we can define the sequence of \gs{s} $S^{\A,M}_t$ by
\begin{itemize}
\item $S^{\A,M}_0=\hidden^G$;
\item $S^{\A,M}_{t+1}=\reveal\left(M,S^{\A,M}_t,\A\left(S^{\A,M}_t \right)\right)$.
\end{itemize}
We say that $\A$ \emph{solves} $M$ if there is a $t$ such that for every $\cell\in M$, we have that $S^{\A,M}_t(\cell)\neq \mine$ and if $S^{\A,M}_t(\cell)=\hidden$ then $M(\cell)=\mine$.

\begin{remark}
Without loss of generality, we made three simplifying assumptions:
\begin{itemize}
\item we do not allow flagging for algorithms here (in real life, they just help the player remember where the mines are). However, we will allow them for \gs{s}, as it will simplify the exposition for ambiguous patterns (see \Cref{sec_patterns});
\item the algorithms we consider do not take the number of mines as input;
\item we only consider deterministic algorithms.
\end{itemize}
The last two assumptions are not restrictive for our purposes\footnote{as long as we don't want to consider random algorithms whose source of randomness is \textbf{not} independent with the mine placement, which would defeat the purpose of the game anyways.}, indeed we will consider the set of all possible algorithms, i.e. all possible moves at any time.
\end{remark}

Let $a_N$ and $b_N$ two integer sequences that both tend to infinity as $N\to\infty$. For each $N$, set  $n=n(N):=a_N\times b_N$ and let $G_n$ the rectangular $a_N\times b_N$ grid. From now on, we forget about $N$ and only talk about $n$, the size of the grid. This will lead to slight abuses of notation, such as saying "as $n\to \infty$", omitting that we are working along a subsequence. In what follows, by "\emph{with high probability}" (abbreviated as "\emph{whp}"), we mean "with probability tending to $1$ as $n\to \infty$".

Given $p=p(n)$, let $\Mnp$ be the mine assignment on $G_n$ such that each cell receives a mine independently with probability $p$.

\section{Ambiguous patterns}\label{sec_patterns}

A \emph{pattern} is a \ma{} on a grid such that there is no mine on the first and last two rows and columns, but there is a mine on the third and third-to-last row and column. Roughly speaking, this way we ensure that patterns are independent (in terms of solving), but that we don't encode the same pattern multiple times by adding empty rows or columns. This definition also rules out patterns that touch the border of the grid. In fact, whp in $\Mnp$ mines will not appear near the border (see \Cref{lem_no_mine_border}). If a subgrid of the \ma{} $M$ matches with a pattern $\Pat$, we say that $M$ \emph{contains an occurrence} of $\Pat$ and we write $\Pat\subset M$.


A \gs{} $S$ is said to be \emph{ambiguous} if it contains at least one hidden cell but no mines and for each of its hidden cells $\cell$, there exists two patterns $\Pat$ and $\Pat'$, both consistent with $S$, such that $\Pat(\cell)=\mine$ and $\Pat'(\cell)=\empty$. Any pattern that is consistent with an ambiguous \gs{} is also said to be ambiguous.

In other words, if, at some point of the game, the player is left with an ambiguous \gs{}, they will not be able to solve it for sure. 

The main result of this section (and the key ingredient of the proof of our main theorem) concerns the smallest possible \ap{s}.

\begin{proposition}\label{prop_patterns}
The only\footnote{remember that we are not allowed to use the border of the grid here, otherwise one could form an ambiguous pattern using only $3$ mines and a corner.} ambiguous patterns with $6$ mines or less  are $\Pat_1$ and $\Pat_2$ depicted in~\Cref{fig_patterns}.
\end{proposition}

\begin{figure}
\center
\includegraphics[width=0.3\textwidth]{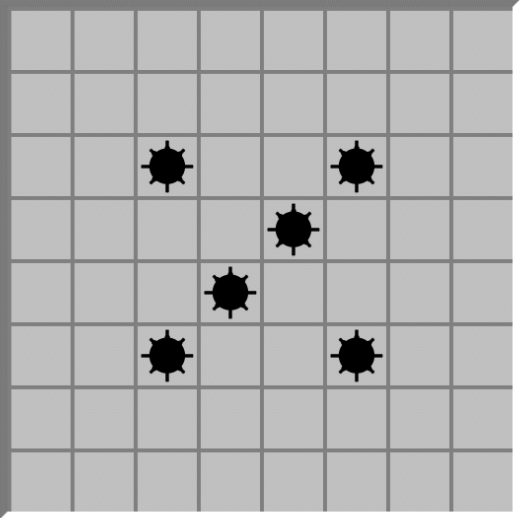}
\includegraphics[width=0.3\textwidth]{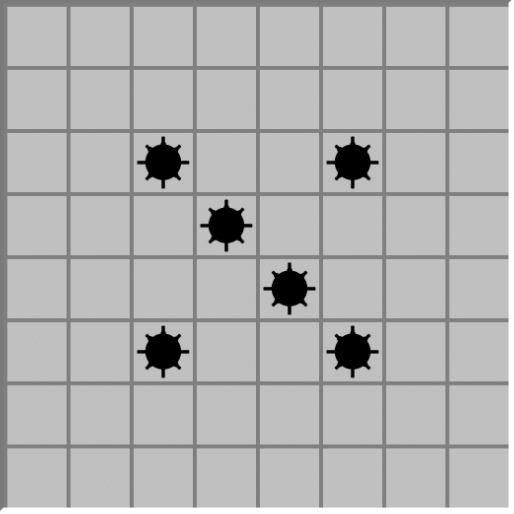}
\includegraphics[width=0.3\textwidth]{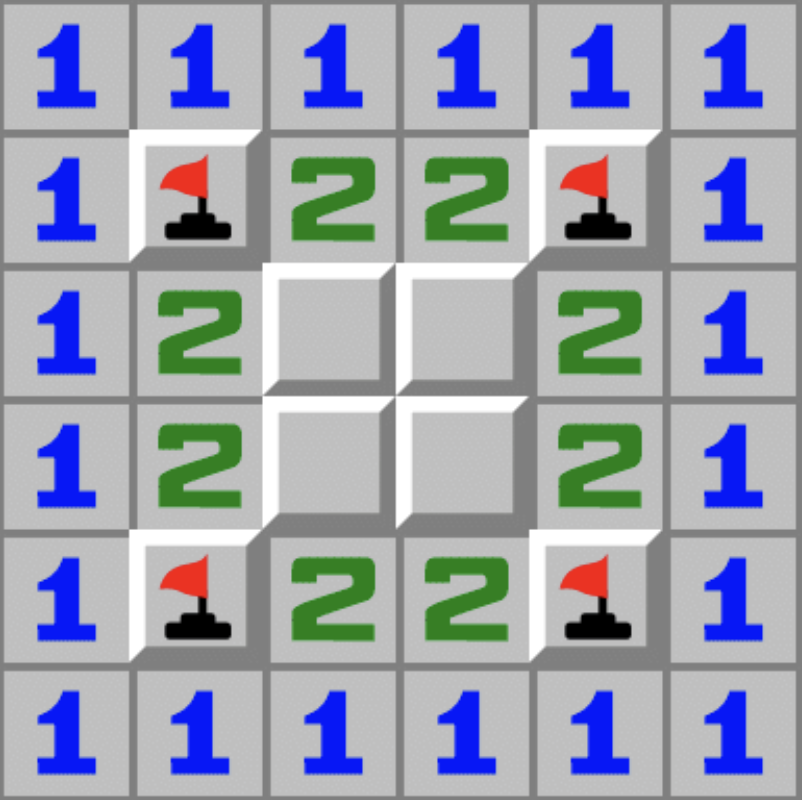}
\caption{The two smallest ambiguous patterns $\Pat_1$ and $\Pat_2$ and their corresponding ambiguous \gs{} $\Smin$.}\label{fig_patterns}
\end{figure}

Let $\Smin$ be the ambiguous \gs{} associated with $\Pat_1$ and $\Pat_2$. Clearly, if the player encounters $\Smin$, they only have a $50\%$ chance to get it right, since $\Pat_1$ and $\Pat_2$ are equally likely in $\Mnp$.  More generally, we have the following result.

\begin{lemma}\label{lem_k_ambiguous_patterns}
Conditionally on $\Mnp$ containing at least $k<<n$ occurrences of $\Pat_1$ and $\Pat_2$, then for any algorithm $\A$
\[\P(\A\text{ solves }\Mnp)\leq 2^{-k}+o(1).\]
\end{lemma}

Before getting to the proof of \Cref{prop_patterns}, let us prove that Minesweeper is not monotone:
\begin{proposition}\label{prop_not_monotone}
Solvability of Minesweeper is not a monotone property.
\end{proposition}
\begin{proof}
We leave it as an exercise to the reader to show that the pattern in \Cref{fig_not_monotone} is not ambiguous, although $\Pat_1$ is ambiguous.

\begin{figure}
\center
\includegraphics[scale=0.3]{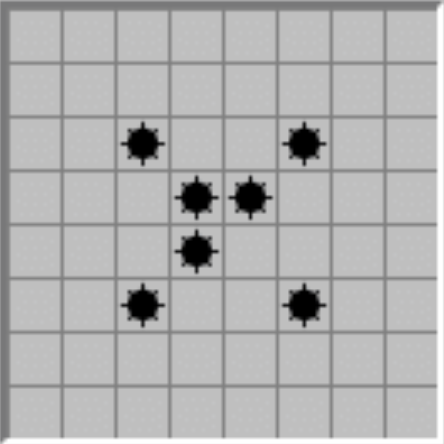}
\caption{Adding a mine to the ambiguous pattern $\Pat_1$ to create a non-ambiguous pattern.}\label{fig_not_monotone}
\end{figure}
\end{proof}

\subsection{Proof of \Cref{prop_patterns}}

Given an ambiguous \gs{} $S$ on the \ma{} $M$, its \emph{envelope} $\env(S)$ is the union of the neighborhoods of its hidden cells (see \Cref{fig_envelope}), i.e. \[\env(S)=\bigcup_{\cell\in S|S(\cell)=\hidden} N(\cell).\]
The \emph{border cells} of $\env(S)$ are the cells $\cell\in\env(S)$ such that $N(\cell)\cap \env(S)<9$ and the \emph{corner cells} are the border cells who have at least two of their orthogonal neighbors outside $\env(S)$.

For a clue cell $\cell\in S$, let
\[\Sc(\cell)=S(\cell)-|\{\cell'\in N(\cell)|S(\cell')=\flag\}|.\]

\begin{figure}
\center
\includegraphics[scale=0.5]{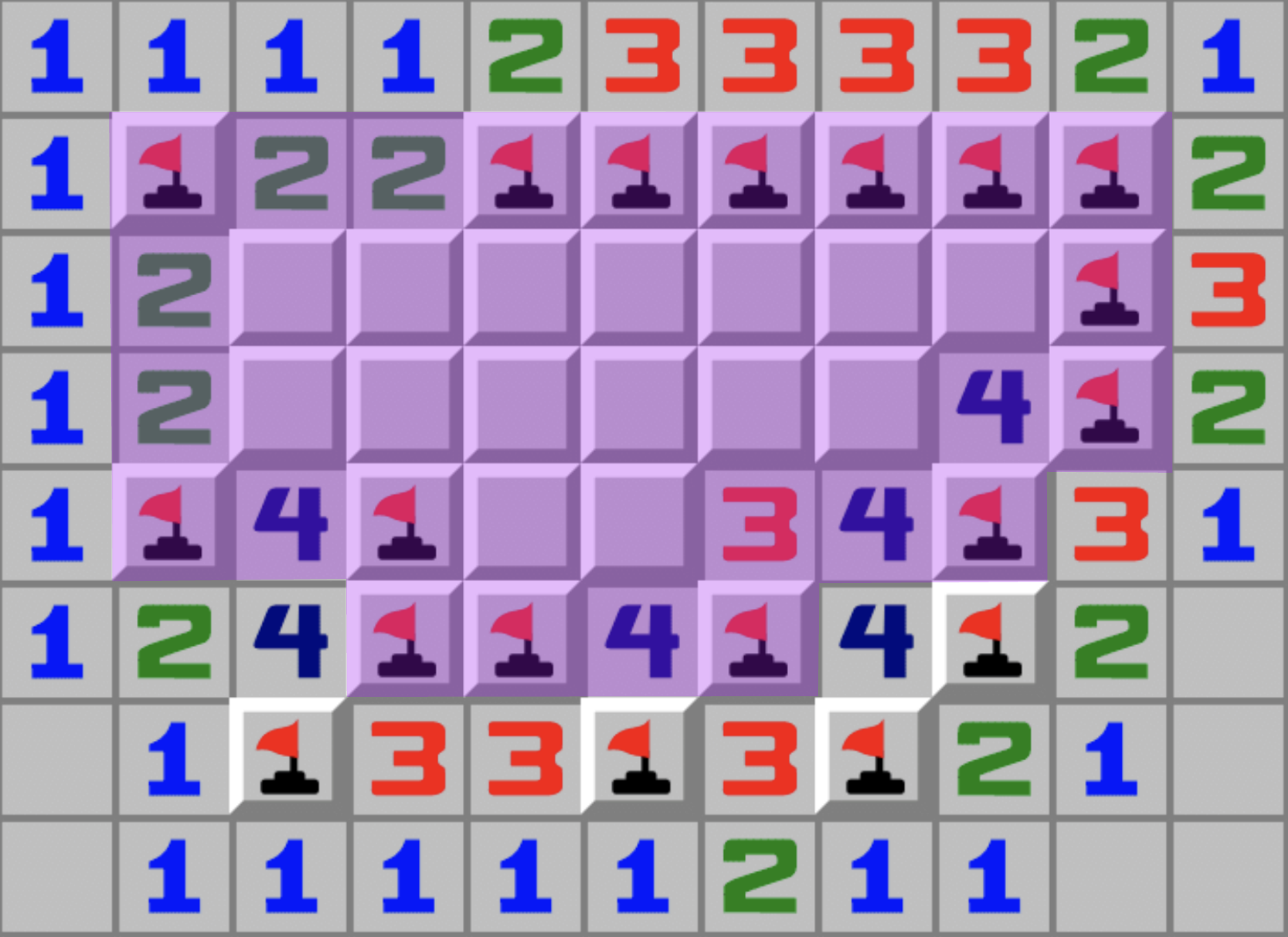}
\caption{An ambiguous \gs{} and its envelope in purple.}\label{fig_envelope}
\end{figure}

We start with two lemmas about envelopes.

\begin{lemma}\label{lem_clues_in_envelopes}
Let $S$ be an ambiguous \gs{}. If $\cell\in\env(S)$ is a clue, then $\Sc(\cell)>0$ and $N(\cell)$ contains at least two hidden cells. 
\end{lemma}

\begin{proof}
By definition of the envelope, $N(\cell)$ contains at least one hidden cell. Assume by contradiction that it contains exactly one, and call it $\cell'$. Then $\Sc(\cell)=0$ or $1$, and its value prescribes whether $\cell'$ hides a mine or not, a contradiction.
\end{proof}

\begin{lemma}\label{lem_corner_flag}
Let $S$ be an ambiguous \gs{}. Border cells of $\env(S)$ cannot be hidden, and any corner cell of $\env(S)$ is a flag.
\end{lemma}

\begin{proof}
The first property is immediate by definition of the envelope.
If $\cell$ is a corner, then by definition of $\env(S)$, $N(\cell)$ contains a unique hidden cell. By \Cref{lem_clues_in_envelopes}, $\cell$ is not a clue hence $S(\cell)=\flag$.
\end{proof}

From now on on and until the end of this subsection,  $\Pat$ is an arbitrary \ap{} with $6$ mines or less, and $S$ is its associated ambiguous \gs{}. We start with a very natural lemma.

\begin{lemma}\label{lem_no_six_flags}
The ambiguous \gs{} $S$ contains $5$ flags or less.
\end{lemma}
\begin{proof}
Assume the contrary, i.e. that $S$ contains $6$ flags and no "hidden mines" (which means that all hidden cells $\cell\in S$ are such that $\Pat(c)=\empty$). $\env(S)$ cannot contain a clue, otherwise, by \Cref{lem_clues_in_envelopes}, $S$ would contain a hidden mine. Therefore, $\env(S)$ contains only flags and hidden cells. Its border has size at least $8$, and by \Cref{lem_corner_flag} it must be all flags, hence $S$ contains strictly more that $6$ flags, a contradiction.
\end{proof}

We continue by restricting the possibilities for $\Pat$.

\begin{lemma}\label{lem_rectangle}
$\env(S)$ is a rectangle.
\end{lemma}

\begin{proof}
By \Cref{lem_corner_flag,lem_no_six_flags}, $\env(S)$ can only have $4$ or $5$ corners. Let us assume by contradiction that $\env(S)$  has $5$ corners. By \Cref{lem_no_six_flags}, all the other cells in the border of $\env(S)$ cannot be flags and therefore are clues. By definition of the envelope, the top row of $\env(S)$\footnote{by this, we mean the highest row of $S$ that contains cells of $\env(S)$.}  has at least one border cell that is not a corner, hence it is a clue, hence the row below it contains a hidden mine of $\Pat$. Apply the same reasoning to the bottom row of $\env(S)$, its leftmost column and its rightmost column. Now, remember that $S$ has already $5$ flags, hence $\Pat$ has only one hidden mine. This would imply that the hidden mine is at the same time on the second and second-to-last row and column of $\env(S)$, hence $\env(S)$ is contained in a $3\times 3$ box, but it also contains this $3\times 3$ box by definition of an envelope, but that contradicts the fact that it has $5$ corners.
\end{proof}

We can finally prove \Cref{prop_patterns}.

\begin{proof}[Proof of \Cref{prop_patterns}]
We know by \Cref{lem_rectangle} that $\env(S)$ is a rectangle, let $w$ and $h$ be its dimensions. We know that $w,h\geq 3$ by definition of the envelope.

First, assume by contradiction that $w=3$. The top $3$ cells of $\env(S)$ are contained in the neighborhood of only one hidden cell, hence by \Cref{lem_clues_in_envelopes}, they are all flags. Same goes for the $3$ bottom cells of $\env(S)$, hence $\env(S)$ contains at least $6$ flags, a contradiction by \Cref{lem_no_six_flags}. Therefore from now on we can assume $w,h\geq 4$.

We will use the following property, that follows directly from \Cref{lem_clues_in_envelopes}:\\
\textit{If $\env(S)$ contains $j$ clue cells with pairwise disjoint neighborhoods, then $S$ contains at least $j$ hidden mines.}

Now, clearly, since $w,h\geq 4$, if the border of $\env(S)$ does not contain $2$ clue cells with disjoint neighborhoods, it must contain at least $6$ flags\footnote{actually, it would even contain at least $9$ flags.}, a contradiction by \Cref{lem_no_six_flags}. Therefore $S$ contains exactly two hidden mines and  $4$ flags (the corners of $\env(S)$).

If $w\geq 6$, one can find four clue cells of $\env(S)$ with pairwise disjoint neighborhoods (the second second-to-last cells of the top and bottom rows), which would imply the existence of at least $4$ hidden mines by the property above, a contradiction.

Therefore $4\leq w,h\leq 5$. Assume finally by contradiction that $w=5$. For $1\leq i\leq 5$, let $a_i$ (resp. $b_i$) be the $i$-th cell of the top row (resp. second row) of $\env(S)$. The only hidden cells in the neighborhood of $a_2$ are $b_2$ and $b_3$, the only hidden cells in the neighborhood of $a_4$ are $b_4$ and $b_3$, and the only hidden cells in the neighborhood of $a_3$ are $b_2$, $b_3$ and $b_4$. Therefore we have the system of equations
\begin{align*}
&\Sc(a_2)=\mathbf{1}_{b_2 \text{ contains a mine}}+\mathbf{1}_{b_3 \text{ contains a mine}}\\
&\Sc(a_3)=\mathbf{1}_{b_2 \text{ contains a mine}}+\mathbf{1}_{b_3 \text{ contains a mine}}+\mathbf{1}_{b_4 \text{ contains a mine}}\\
&\Sc(a_4)=\mathbf{1}_{b_3 \text{ contains a mine}}+\mathbf{1}_{b_4 \text{ contains a mine}}
\end{align*}
therefore the locations of the mines in the second row would be determined, a contradiction since we have an ambiguous \gs{}.

This shows that $w=h=4$ and from there it is direct to deduce that $S=\Smin$.

\end{proof}

\subsection{Discussion}
The proof of \Cref{prop_patterns} is quite tedious and crucially relies on the assumption that the number of mines is bounded by $6$ (in fact, even \Cref{lem_no_six_flags}, which seems obvious at first sight, is not true anymore when considering patterns with  more mines, see \Cref{fig_8_mines_8_flags}). 
\begin{figure}\center
\includegraphics[width=0.3 \textwidth]{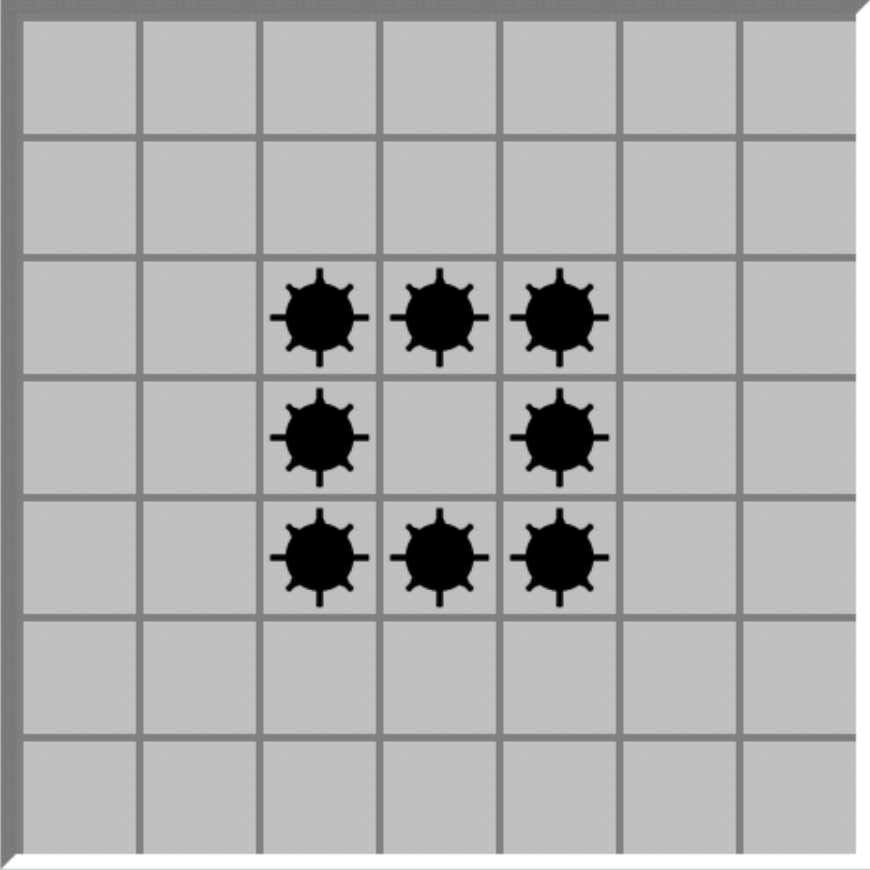}
\includegraphics[width=0.3 \textwidth]{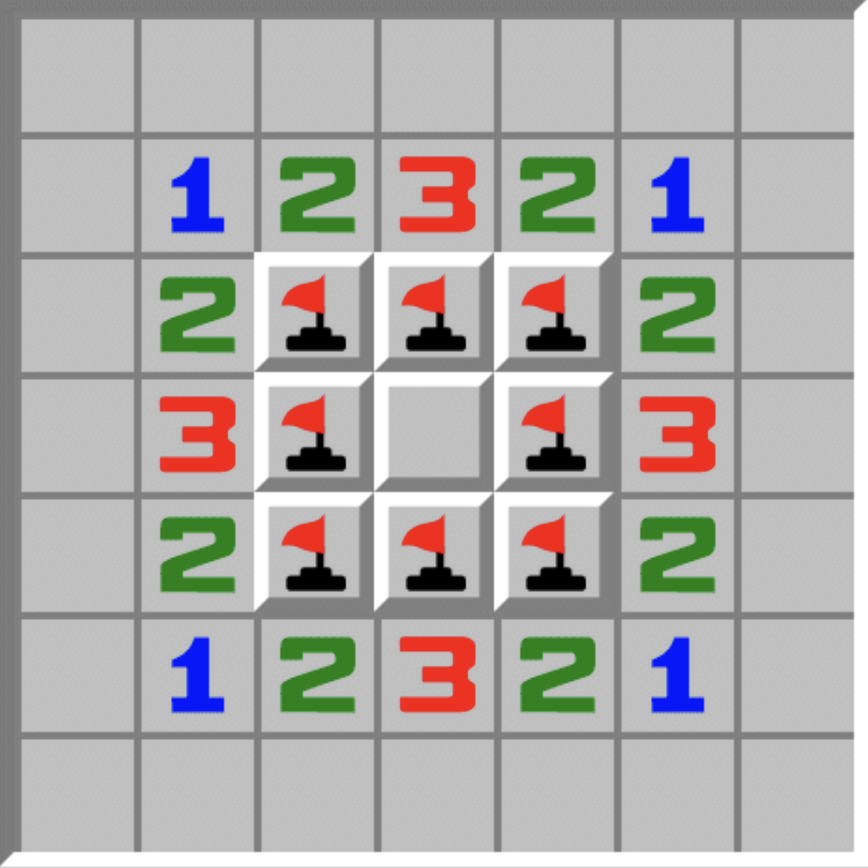}
\caption{An ambiguous pattern with $8$ mines, and its associated ambiguous \gs{} with $8$ flags}\label{fig_8_mines_8_flags}
\end{figure}

In particular, this implies that the proof of \Cref{prop_patterns} does not give us much insight about what ambiguous patterns look like in general.

\begin{open}
Can we characterize/enumerate ambiguous \gs{s}/patterns ?
\end{open}

The hardness of the inference problem rules out a simple characterization, and exact enumeration should not be reasonably expected either, however maybe some asymptotic properties can be proven. We leave this question open, and we finish this section with two more nice ambiguous \gs{s}.

\begin{figure}
\center
\includegraphics[scale=0.37]{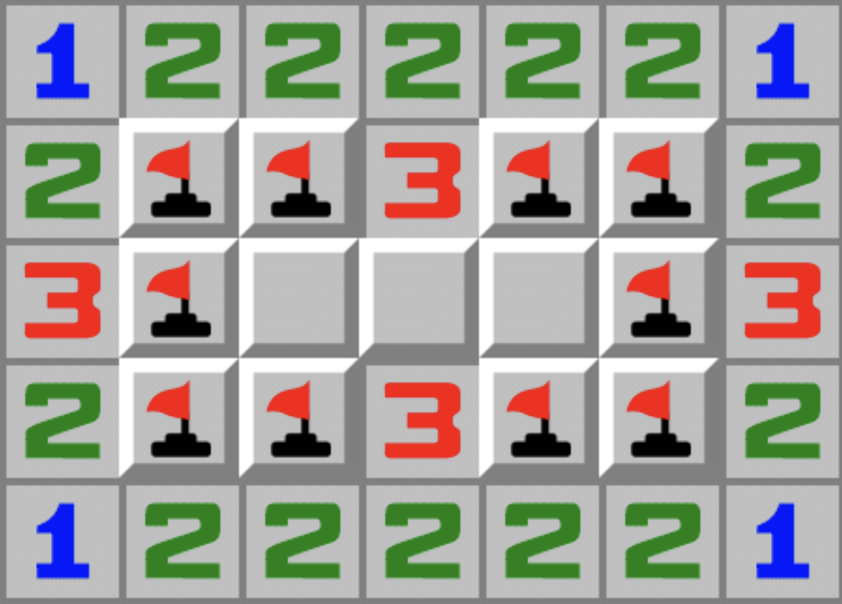}\quad
\includegraphics[scale=0.5]{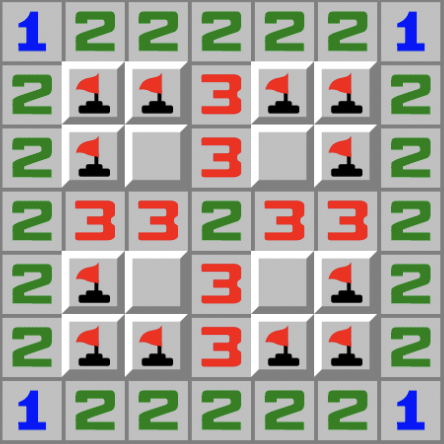}
\caption{Two more nice ambiguous \gs{s}.}
\end{figure}

\section{Proof of the main theorem}\label{sec_proof}
In this section, we prove \Cref{main_thm}.
\subsection{Preliminaries}
We gather here all the technical lemmas. The first one is immediate by definition of the grid $G_n$.

\begin{lemma}\label{lem_no_mine_border}
If $p=o(1)$, there is no mine at distance $100$ or less of the border of the grid in $\Mnp$.
\end{lemma}

We now show that when the mine density is small enough, mines do not concentrate too much in small regions.

\begin{lemma}\label{lem_bad_box}
Let $k=6$ or $7$. If $p<<n^{-1/k}$, then, the expected number of $100\times 100$ subgrids of $\Mnp$, containing $k$ mines or more is $o(1)$.
\end{lemma}

\begin{proof}
Given a $100\times 100$ subgrid of $\Mnp$, the probability that it contains $k$ mines or more is bounded above by $\binom{10000}{k}p^k$. Hence the expected number of $100\times 100$ subgrids of $\Mnp$, containing $k$ mines or more is bounded by $\binom{10000}{k}np^k=o(1)$.
\end{proof}

Let $S_{n,p}$ be the \gs{} of $\Mnp$ after revealing the cell in the upper left corner, then iteratively revealing any cell that is in the neighborhood of a cell with value $0$\footnote{something that Minesweeper does automatically for the player in the computer version.}.

An \emph{island} of $S_{n,p}$ is a connected component of cells such that $S_{n,p}(\cell)\neq 0$ (see \Cref{fig_island}). In other words, each island is independent of the others, and solving $\Mnp$ is equivalent to solving each island separately. 

\begin{figure}
\center
\includegraphics[scale=1]{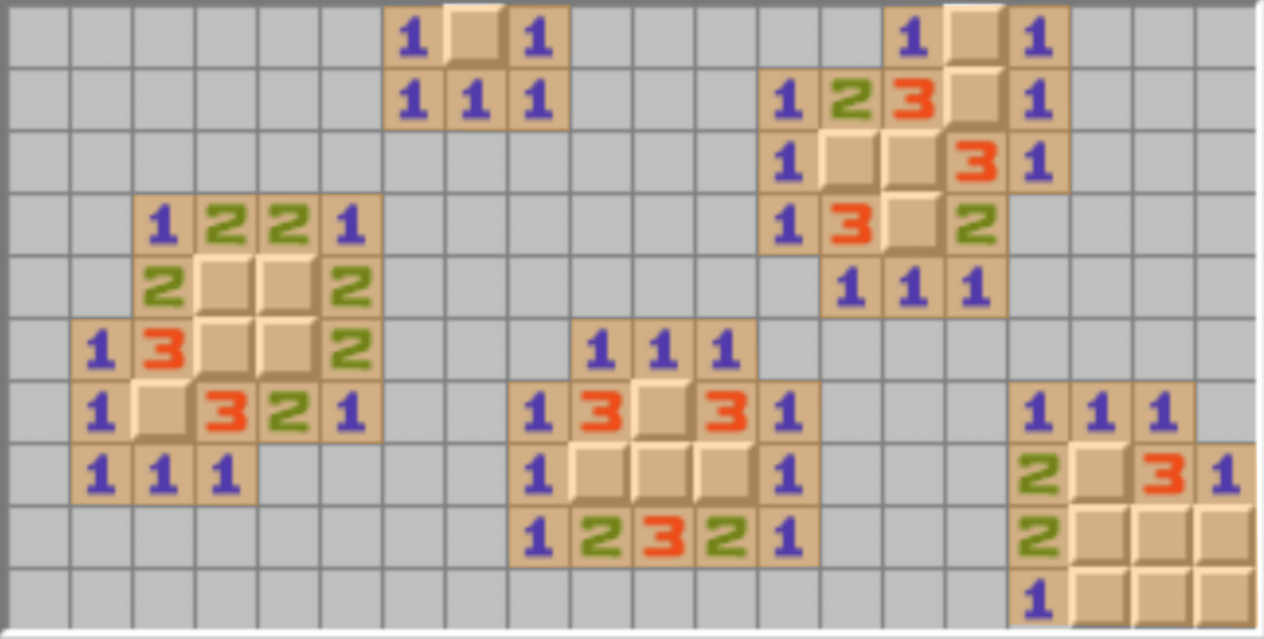}
\caption{A grid state and its islands in orange.}\label{fig_island}
\end{figure}

\begin{lemma}\label{lem_islands}
Let $k=6$ or $7$. If $p<<n^{-1/k}$, no island of $S_{n,p}$ contains more than $k-1$ mines, and every island fits in a $100\times 100$ box.
\end{lemma}

\begin{proof}
For any island $I\subset S_{n,p}$, we can build an "adjacency graph"
$A$ where the vertices are labeled by the mines i.e $V(A)=\{\cell\in I|\Mnp(\cell)=\mine\}$ and $\cell\sim\cell'$ iff $d(\cell,\cell')\leq 3$. It is easily checked that for any island, its adjacency graph is connected.

Now, let $I\subset S_{n,p}$ be an island containing at least $k$ mines and $A$ its adjacency graph. 
Since $A$ is connected, one can find $k$ of its vertices such that the induced subgraph is still connected. Let $A'$ be this graph, it has diameter $\leq 6$, hence the distance (in the grid) between two cells of $V(A')$ is at most $18$. Therefore these $k$ cells fit in a $100\times 100$ subgrid of $\Mnp$. Which means that if $S_{n,p}$ contains an island with $k$ mines or more, $\Mnp$ contains a $100\times 100$ subgrid with at least $k$ mines. By \Cref{lem_bad_box}, this happens with probability $o(1)$, which concludes the proof.
\end{proof}

Finally, we show that when the mine density is large enough, ambiguous patterns emerge.

For any \ma{} $M^*$ on $G_n$, define $M^*_p$ in the following way: for every cell $\cell\in  G_n$, if $M^*(\cell)=\mine$ then $M^*_p(\cell)=\mine$, otherwise $M^*_p(\cell)=\mine$ with probability $p$.

\begin{lemma}\label{lem_emergence_patterns}

If $M^*$ contains less than $n^{0.9}$ mines and $\cut<<p<1-\eps$ for some $\eps>0$, then there exists $C>0$ depending only on $\eps$ such that for $n$ large enough, $M^*_p$ contains more than $Cnp^6$ occurrences of $\Pat_1$ and $\Pat_2$ with probability greater than $ 1-\exp\left(-\frac{Cnp^6}{4}\right)$
\end{lemma}

\begin{proof}
Set $C=\eps^{58}$. In $M^*$, for $n$ large enough one can find at least $n/100$ pairwise disjoint subgrids of dimensions $8\times 8$ not containing any mines (tile $G_n$ with $n/64-o(n)$ boxes of dimensions $8\times 8$ and throw away those who contain a mine of $M^*$). In each of those, the probability that the mines in $\Mnp$ form the either the pattern $\Pat_1$ or the pattern $\Pat_2$ is $2p^6(1-p)^{58}\geq 2Cp^6$. Since what happens in each of the $8\times 8$ subgrids is independent, we can conclude with Chernoff bounds.
\end{proof}

The statement of \Cref{lem_emergence_patterns} can be confusing, one should first think of it with $M^*=\empty^{G_n}$ and therefore $M^*_p=\Mnp$, which will be enough for the proof of \Cref{main_thm}. We will only use \Cref{lem_emergence_patterns} in full generality in \Cref{prop_monotonicity}.

\subsection{Above criticality}
Here, $p>>n^{-1/6}$. If $p=1-o(1)$ then any algorithm's first choice for a cell hits a mine whp and loses. Hence in what follows we can assume there exists $\eps>0$ such that $p<1-\eps$. Since $np^6\to \infty$, by~\Cref{lem_emergence_patterns} (for $M^*=\empty^{G_n}$) whp the number of occurrences of both $\Pat_1$ and $\Pat_2$ in $\Mnp$ tends to infinity, and we conclude by \Cref{lem_k_ambiguous_patterns} that
\[\max_{\text{algorithms }\A}\P(\A\text{ solves }\Mnp)=o(1).\]

\subsection{Below criticality}\label{sec_solve}

Here we take $p=o(\cut)$. First, pick any arbitrary (valid) inference algorithm, i.e. an algorithm that selects a safe cell if it finds one, and gives up otherwise. We describe our algorithm that solves $\Mnp$ whp:

\begin{enumerate}
\item Reveal the upper left corner cell;
\item As long as there exists a hidden cell in the neighborhood of a revealed cell with value $0$, pick one arbitrarily and reveal it;
\item Record where the islands are in the grid;
\item For each island, if it fits in a $100\times 100$ box, apply the inference algorithm to it, otherwise, give up.
\end{enumerate}

By \Cref{lem_islands,lem_no_mine_border,prop_patterns} with high probability no step will fail, i.e. the first cell revealed will not be a mine, all islands will be small enough and away from the border and the inference algorithm will solve each of them.

Speaking of runtime, step 2 can be performed in linear time if one keeps a list of cells to reveal next: each time a cell with label $0$ is revealed, add its hidden neighbors to the "to reveal" list (redundancies do not matter since every cell can be added at most $8$ times to the list). Step 3 can also be performed in linear time by scanning the grid linearly, and step 4 takes constant time per island (no matter how bad the inference algorithm, it is applied to a bounded size grid). Therefore our algorithm runs in linear time.

\subsection{At criticality}

Take $p\sim c\cut$. We sketch the proof of this section and leave the details as an exercise to the reader. The expected number of occurrences of $\Pat_1$ (resp. $\Pat_2$) is $c^6+o(1)$. By \Cref{lem_bad_box} for $k=7$ and  the Stein-Chen method \cite{SteinChen}, we can show that the number of occurrences of $\Pat_1$ and $\Pat_2$ converge to independent Poisson laws of parameter $c^6$. By \Cref{lem_islands} for $k=7$, whp no island of $S_{n,p}$ contains more than $6$ mines. Therefore, with probability $\exp(-2c^6)+o(1)$, no islands of $S_{n,p}$ contain an ambiguous pattern, hence the same algorithm as in the previous section solves $\Mnp$. But with probability $1-\exp(-2c^6)+o(1)$ there is at least an ambiguous pattern and by \Cref{lem_k_ambiguous_patterns}
\[\max_\A \P(\A \text{ solves }\Mnp)\leq 1-\frac{1-\exp(-2c^6)}{2}+o(1).\]

\section{The random process}\label{sec_process}

We define the process $M_n^t$ for $0\leq t\leq n$ in the following way: $M_n^0=\empty^{G_n}$, and for each $t$, $M_n^{t+1}$ is built out of $M_n^t$ by picking uniformly at random one of its empty cells and adding a mine in it. All steps are pairwise independent. We let $\tau$ be the smallest $t$ such that $M_n^t$ contains either $\Pat_1$ or $\Pat_2$ (and we set $\tau=\infty$ if such a $t$ does not exist).

\subsection{Hitting time property}

The unsolvability of $M_n^t$ arises with the first occurrence of $\Pat_1$ or $\Pat_2$.

\begin{theorem}\label{thm_hitting_time}
Whp the algorithm of \Cref{sec_solve} solves $M_n^t$ for all $0\leq t<\tau$.
\end{theorem}

The idea is that whp, the first ambiguous pattern to appear will be $\Pat_1$ or $\Pat_2$. 
By the same reasoning as in  \Cref{lem_islands,sec_solve}, \Cref{thm_hitting_time} is a direct consequence of the following result.

\begin{proposition}\label{prop_tau_subgrid}
Whp, for all $t<\tau$, no $100\times 100$ subgrid of $M_n^t$ contains more than $6$ mines.
\end{proposition}

This property is clearly monotone, hence if we set $\cutoff=n^{71/84}$ (the important thing being that $5/6<71/84<6/7$), it suffices to show the two following properties to imply \Cref{prop_tau_subgrid}.
\begin{lemma}\label{lem_tau_vs_bad_box}
Whp, $M_n^{\cutoff}$ contains an occurrence of $\Pat_1$ and no $100\times 100$ subgrid of $M_n^{\cutoff}$ contains more than $6$ mines. 
\end{lemma}

\begin{proof}
We start with the second property. Let $E$ be the set of tuples of $7$ cells of $G_n$ such that they all fit in a $100\times 100$ subgrid of $G_n$. The cardinality of $E$ is easily bounded by $\binom{10000}{7}n$.

If a $100\times 100$ subgrid of $M_n^{\cutoff}$ contains $7$ mines or more, then there must exist $(\cell_i)_{i=1…7}\in E$ such that for all $1\leq i\leq 7$, $\cell_i$ gets added at some time $t_i\leq \cutoff$. For each element of $E$, this happens with probability bounded by
\[7!\binom{\cutoff}{7}\left(\frac{2}{n}\right)^7\leq 2^7 n^{-13/12},  \]
and we conclude by a union bound.

To prove the first property, we use the fact that $M_n^t$ has the same law as $\Mnp$ conditioned to have exactly $t$ mines, hence
\begin{align*}
\P(\Pat_1\not\subset M_n^{\cutoff})&= \P(\Pat_1\not\subset \Mnp|\Mnp\text{ contains exactly $\cutoff$ mines})\\
&=\frac{\P(\Pat_1\not\subset \Mnp \text{ and }\Mnp\text{ contains exactly $\cutoff$ mines})}{\P(\Mnp\text{ contains exactly $\cutoff$ mines})}\\
&\leq \frac{\P(\Pat_1\not\subset \Mnp)}{\P(\Mnp\text{ contains exactly $\cutoff$ mines})}
\end{align*}

Let us take $p=\frac{\cutoff}{n}$.
On the one hand, by \Cref{lem_emergence_patterns} applied to $M^*=\empty^{G_n}$, $\P(\Pat_1\not\subset \Mnp)\leq \exp(-n^{1/1000})$ for $n$ large enough. On the other hand, the number of mines in $\Mnp$ is a sum of Bernoulli random variables therefore
\begin{align*}
\P(\Mnp\text{ contains exactly $\cutoff$ mines})&=\binom{n}{\cutoff}\left(\frac{\cutoff}{n}\right)^{\cutoff}\left(1-\frac{\cutoff}{n}\right)^{n-\cutoff}\\
&\sim \frac{1}{\sqrt{2\pi\cutoff}}\frac{n^n}{\cutoff^{\cutoff} (n-\cutoff)^{n-\cutoff}}\left(\frac{\cutoff}{n}\right)^{\cutoff}\left(1-\frac{\cutoff}{n}\right)^{n-\cutoff}\\
&=\frac{1}{\sqrt{2\pi\cutoff}}.
\end{align*}
Thus $\P(\Pat_1\not\subset M_n^{\cutoff})=o(1)$ and we are done.
\end{proof}

\subsection{Almost sure monotonicity}

Although we saw in \Cref{prop_not_monotone} that solvability of Minesweeper is not a monotone property, i.e. during the process, a mine could appear and "destroy" an ambiguous pattern that was created before, this most likely never happens. Here we restrict ourselves to running the process until time $t=n/2$, after that, the first mine that is revealed has at least a chance $1/2$ of being a mine anyways. We will also condition on $\tau$ being sublinear in $n$, something that happens whp by \Cref{lem_tau_vs_bad_box}.

\begin{proposition}\label{prop_monotonicity}
Conditionally on $\tau<n^{0.9}$ and on $M_n^\tau$, whp for all $\tau<t<n/2$, $M_n^t$ contains at least one occurrence of $\Pat_1$ or $\Pat_2$.
\end{proposition}

\begin{proof}
Let $B^*$ the $8\times 8$ subgrid of $M_n^\tau$ that contains the pattern $\Pat_1$ or $\Pat_2$. For all $1\leq t\leq n^{0.9}$, at step $\tau+t$ there are more than $n/2$ "free cells" so the new mine is added outside of $B^*$ with probability greater than $1-\frac{128}{n}$. Since each step is independent, the probability that no mine is added to $B^*$ between times $\tau+1$ and $\tau+n^{0.9}$ is 
\[\left(1-\frac{128}{n}\right)^{n^{0.9}}=1-o(1),\]
and thus whp, for all $\tau<t<\tau+n^{0.9}$, $M_n^t$ contains at least one occurrence of $\Pat_1$ or $\Pat_2$ (at $B^*$).

Now we focus on $n^{0.9}<t<n/2$. We will use the same reasoning as in \Cref{thm_hitting_time} hence we only sketch the proof. If we set $M^*=M_n^\tau$, notice that $M_n^{\tau+t}$ is $M^*_p$  conditioned on having exactly $\tau+t$ mines.
Fix a $t$ with $n^{0.9}<t<n/2$ and set $p=\frac{t}{n-\tau}$, we have
\[\P(\Pat_1,\Pat_2\not \subset M_n^{\tau+t})\leq \frac{\P(\Pat_1,\Pat_2 \not \subset M^*_p)}{\P(M^*_p\text{ contains exactly $\tau+t$ mines})}.\]

On the one hand, the probability that $M^*_p$ has exactly $\tau+t$ mines is equivalent to $\frac{1}{\sqrt{2\pi t}}\geq \frac{1}{\sqrt{2\pi n}}$.
On the other hand, by \Cref{lem_emergence_patterns},  $M^*_p$ contains no occurrence of $\Pat_1$ or $\Pat_2$ with probability smaller than $\exp(-Cnp^6)\leq \exp(-Cn^{0.4})$. Therefore, for $n$ large enough and conditionally on $M_n^\tau$, for every $n^{0.9}<t<n/2$, $M_n^{\tau+t}$ contains no occurrence of $\Pat_1$ or $\Pat_2$ with probability smaller than $\exp(-n^{0.3})$, and we conclude with a union bound that for all $\tau+n^{0.9}<t<n/2$, $M_n^t$ contains at least one occurrence of $\Pat_1$ or $\Pat_2$.
\end{proof}

\bibliography{bibli}{}
\bibliographystyle{plain}

\end{document}